\documentclass[12pt]{article}
\usepackage{amsthm,amssymb,amsmath,graphicx,cite}
\usepackage{algorithmic, algorithm}
\usepackage[margin = 1.15in]{geometry}
\usepackage{multirow}
\usepackage{float}
\usepackage{fullwidth}
\usepackage{placeins}
\usepackage[margin=1cm]{caption}
\usepackage{makecell}
\usepackage{multirow}
\usepackage{float}
\setcellgapes{4pt}
\usepackage{subfig}
\usepackage{color}

\newcommand{\diam}{\mathsf{diam}}

\newcommand{\Oh}{{\mathcal O}}

\newcommand{\ip}[2]{\left\langle #1 , #2 \right\rangle}    
\newcommand{\R}{{\mathbb R}}

\newcommand{\dom}{{\mathrm{dom}}}

\newcommand{\D}{{\mathcal{D}}}
\DeclareMathOperator*{\argmin}{arg\,min}

\newtheorem{lemma}{Lemma}
\newtheorem{theorem}{Theorem}
\newtheorem{proposition}{Proposition}

\newcommand{\subopt}{{\mathsf{subopt}}}
\newcommand{\gap}{{\mathsf{gap}}}

\title{Affine invariant convergence rates of the conditional gradient method}
\author{
Javier F. Pe\~na\thanks{Tepper School of Business,
Carnegie Mellon University, USA, {\tt jfp@andrew.cmu.edu}}
}

\begin{document}

\maketitle

\begin{abstract}
We show that the conditional gradient method for the convex composite problem \[\min_x\{f(x) + \Psi(x)\}\] generates primal and dual iterates with a duality gap converging to zero provided a suitable {\em growth property} holds and the algorithm makes a judicious choice of stepsizes.  The rate of convergence of the duality gap to zero ranges from sublinear to linear depending on the degree of the growth property.  The growth property and convergence results depend on the pair  $(f,\Psi)$ in an affine invariant and norm-independent fashion. 
\end{abstract}
\section{Introduction}
We consider the conditional gradient method for the 
 composite minimization problem
\[
\min_x \{f(x)+ \Psi(x)\}
\]
where $f:\R^n\rightarrow \R\cup\{\infty\}$ and $\Psi:\R^n\rightarrow \R\cup\{\infty\}$ are closed convex functions, $f$ is differentiable on an open set containing $\dom(\Psi)$, and there is an available oracle for the mapping $$g\mapsto \argmin\{\ip{g}{y}+\Psi(y)\}.$$  The traditional format of the conditional gradient method  (also known as the Frank-Wolfe algorithm) as discussed in~\cite{FreuG16,Jagg13} corresponds to the case when $\Psi$ is the indicator function of a compact convex set in $\R^n$.  Some of the most attractive features of the conditional gradient method are its affine invariance, norm-independence, and lack of reliance on any projection mapping.  These features stand in stark contrast to those of most other first-order methods.  It is therefore natural to seek convergence results for the conditional gradient method that are affine invariant and norm-independent as any otherwise convergence result would typically be loose due to dependence on spurious factors.

We show that the conditional gradient method generates primal and dual iterates with a duality gap converging to zero provided a suitable {\em growth property} holds (see~\eqref{eq.growth.cond}) and the algorithm makes a judicious choice of stepsizes (see~\eqref{eq.stepsize} and~\eqref{eq.stepsize.relax}).    The rate of convergence of the duality gap to zero ranges from sublinear to linear (see Theorem~\ref{thm.main} and Theorem~\ref{prop.conv.gral}) depending on the degree of the growth property.  The growth property and convergence results depend on the pair  $(f,\Psi)$ in an affine invariant and norm-independent fashion.  In fact, our main developments do not rely on any norms at all.

Our work is inspired by recent developments of Nesterov~\cite{Nest18}, Ghadimi~\cite{Ghad19}, and Kerdreux et al~\cite{Kerd21}.   
Under the assumption that the domain of $\Psi$ is bounded, Nesterov~\cite{Nest18} established a $\Oh(1/k^\nu)$ convergence rate for the conditional gradient method when $\nabla f$ is $\nu$-H\"older continuous. Nesterov~\cite{Nest18} also showed the stronger convergence rate $\Oh(1/k^{2\nu})$ under the additional assumption that $\Psi$ is strongly convex.  Ghadimi~\cite{Ghad19} 
showed an even stronger linear rate of convergence when $\nabla f$ is Lipschitz continuous and $\Psi$ is strongly convex for a judicious choice of stepsizes.  On the other hand, Kerdreux et al.~\cite{Kerd21} considered the case when $\nabla f$ is Lipschitz continuous and $\Psi$ is the indicator function of a uniformly convex set.  Kerdreux et al.~\cite{Kerd21} showed that the convergence rate ranges from sublinear to linear depending on the degree of uniform convexity of the set.
The results in~\cite{Ghad19,Kerd21,Nest18} are all norm-dependent since they rely on H\"older continuity and uniform convexity.  The norm-dependence of the results in~\cite{Ghad19,Kerd21,Nest18} in turn implies that they are not affine invariant when the norms are not affine invariant, which is typically the case.

The affine invariance and norm-independence properties of the conditional gradient method suggest that there should be an affine invariant and norm-independent approach to obtain the sublinear to linear spectrum of convergence rates.  The central contribution of this paper is to show that this is indeed the case.  To that end, we establish general convergence results (Theorem~\ref{thm.main} and Theorem~\ref{prop.conv.gral}) provided a {\em growth property}~\eqref{eq.growth.cond} holds  and the stepsizes are judiciously chosen.  The growth property~\eqref{eq.growth.cond} can be seen as an extension of the {\em curvature constant} proposed by Jaggi~\cite{Jagg13}.  It also has a flavor similar to that of the {\em relative smoothness} and {\em relative continuity} properties in the context of Bregman proximal methods as discussed in~\cite{BausBT16,Lu17,LuFN18,Tebo18,VanN16,VanN17}. 
Our main  convergence results (Theorem~\ref{thm.main} and Theorem~\ref{prop.conv.gral}), obtained via the growth property~\eqref{eq.growth.cond},  are stated in terms of the duality gap between the primal and dual iterates that the algorithm generates.  
We also  develop parallel results (Theorem~\ref{thm.main.redux} and Theorem~\ref{prop.conv.gral.redux}) on the suboptimality gap via a relaxation~\eqref{eq.weak.growth.cond} of the growth property~\eqref{eq.growth.cond}.

Aside from its natural compatibility with the properties of the conditional gradient method, the affine invariance and norm independence of our approach yields convergence results that are  sharper and more general than previous approaches that do not have these two fundamental properties.  More precisely, as we detail in Section~\ref{sec.growth.examples}, Section~\ref{sec.weak}, and Section~\ref{sec.final}, our convergence results not only subsume some of the norm-dependent convergence results in~\cite{GarbH15,Ghad19,Kerd21,Nest18,XY18} but are also at least as sharp as these previous norm-dependent results for {\em any} choice of norms.  Put differently, the affine invariance and norm-independence of our main convergence results automatically circumvent any possible spurious overhead due to a particular choice of norms or affine scaling.  It is also noteworthy that our main results (Theorem~\ref{thm.main}, Theorem~\ref{prop.conv.gral}, Theorem~\ref{thm.main.redux}, and Theorem~\ref{prop.conv.gral.redux}) rely only on straightforward stepsize procedures that are adaptive to the typically unknown parameters associated to the growth or weak growth properties.  This stands in sharp contrast to some of the algorithmic schemes in~\cite{GarbH15,Ghad19,Kerd21,Nest18,XY18} that rely on knowledge of specific parameters like Lipschitz continuity constants.

This article is related to a number of papers in the rich and increasingly growing literature on the conditional gradient method.  Our primal-dual approach, with emphasis on bounding the duality gap for the conditional gradient iterates, is similar in spirit to the approaches used by Bach~\cite{Bach15}, Freund and Grigas~\cite{FreuG16}, and Jaggi~\cite{Jagg13}.  Like all of~\cite{Bach15,FreuG16,Jagg13}, our main results provide upper bounds on the duality gap between the (primal) conditional gradient iterates and some canonical dual iterates that the algorithm also generates.
However, our central goal and main results are fundamentally different from those in~\cite{Bach15,FreuG16,Jagg13}.  Bach~\cite{Bach15} focuses primarily on the equivalence between the conditional gradient and mirror descent algorithms.  
All of the convergence results in~\cite{Bach15,FreuG16,Jagg13} are of order $\Oh(1/k)$ or some minor variations.  
By contrast, our main focus is the spectrum of convergence rates of the conditional gradient method ranging from sublinear to linear.  The latter goal is also pursued in the recent articles~\cite{Ghad19,Kerd21,Nest18,XY18} albeit in a norm-dependent fashion.  The articles~\cite{KerdDP22,RZ20} also establish a spectrum of convergence rates ranging from sublinear to linear but for more elaborate variants of the conditional gradient method that include  various types of search directions, restart procedures, and/or are restricted to specific types of domains.  
As we detail in Section~\ref{sec.growth.examples}, Section~\ref{sec.weak}, and Section~\ref{sec.final}, our affine invariant and norm independent results subsume, sharpen, and extend most of the norm-dependent results previously derived in~\cite{GarbH15,Ghad19,Kerd21,Nest18,XY18}.
Our developments, driven by affine-invariance, are in the same spirit as those of Kerdreux et al.~\cite{KerdL21}.  However, the approach in~\cite{KerdL21} applies only to the case when $\Psi$ is the indicator function of a strongly convex set and focuses on an affine-invariant {\em directional smoothness} property that ensures linear convergence.  By contrast, our work applies to general $\Psi$ and relies on the growth property~\eqref{eq.growth.cond} that leads to convergence rates covering the entire range from sublinear to linear.  

The remaining sections of the paper are organized as follows.  Section~\ref{sec.main} presents our main developments, namely the affine invariant and norm-independent growth property~\eqref{eq.growth.cond} and convergence results for the conditional gradient method.  For ease of exposition, we include two convergence results.  First, Theorem~\ref{thm.main} relies on the assumption that the stepsizes are chosen via the exact line-search~\eqref{eq.stepsize}.  Theorem~\ref{prop.conv.gral} shows that the convergence rates in~Theorem~\ref{thm.main} continue to hold, albeit with some more complicated constants, for the more flexible and easily implementable choice of stepsizes~\eqref{eq.stepsize.relax}.  
Section~\ref{sec.growth.examples} details some classes of problems that satisfy the growth property~\eqref{eq.growth.cond}.  As a byproduct, we recover, sharpen, and extend some of the norm-dependent convergence rates established in~\cite{Ghad19,Kerd21,Nest18}.  
Section~\ref{sec.weak} presents developments parallel to those in Section~\ref{sec.main} but for the {\em suboptimality gap} by relying on the {\em weak} version~\eqref{eq.weak.growth.cond} of the growth property~\eqref{eq.growth.cond}.
Section~\ref{sec.final} presents a more flexible and general {\em local} growth property that in turn yields local convergence results for the conditional gradient method.

\section{Growth property and affine invariant convergence}
\label{sec.main}

Consider the composite minimization problem
\begin{equation}\label{eq.problem}
\min_x \{f(x)+ \Psi(x)\} 
\end{equation}
where both $f:\R^n\rightarrow \R\cup\{\infty\}$ and $\Psi:\R^n\rightarrow \R\cup\{\infty\}$ are closed convex functions.  
Algorithm~\ref{algo.CG} describes the conditional gradient algorithm for~\eqref{eq.problem}. The algorithm relies on the following assumptions about~\eqref{eq.problem}:
\begin{itemize}
\item[A1.] The function $f$ is differentiable on an open set containing $\dom(\Psi)$. 
\item[A2.] A point in the following set of minimizers exists and is computable for all $x\in \dom(f)$
\begin{equation}\label{eq.oracle}
\argmin_s\{\ip{\nabla f(x)}{s} + \Psi(s)\}.
\end{equation}
\end{itemize}
The format of problem~\eqref{eq.problem} is the same as that considered in~\cite{Nest18,Ghad19}.  The  traditional format of the conditional gradient method, as discussed in~\cite{FreuG16,Jagg13}, corresponds to the special case when $\Psi$ is the indicator function $\Psi=\delta_C$ of some closed convex set $C\subseteq\R^n$.  
A fundamental difference between the conditional gradient method and other popular first-order methods for~\eqref{eq.problem}, like proximal gradient or more general Bregman proximal gradient methods, is the reliance on the linear oracle~\eqref{eq.oracle} as opposed to the reliance on a proximal mapping for the function $\Psi$.  
This fundamental difference has two profound implications.  First, the conditional gradient method is preferable for a problem of the form~\eqref{eq.problem} where the proximal mapping for $\Psi$ is unavailable or prohibitively expensive but the linear oracle~\eqref{eq.oracle} is available.  This is the case in a large class of problems in data science such as those discussed in~\cite{Bach15,BredCFR22,BredLM09,FreuG16,GarbKS21,Ghad19,HarcJN15,Jagg13,YuZS17}.  In addition to the popular setting where $\Psi$ is the indicator function of a closed convex set, the articles~\cite{Bach15,BredCFR22,BredLM09,GarbKS21,KuniW22,LiD22,MineF81,YuZS17} also discuss problems of the form~\eqref{eq.problem} where $\Psi$ is not an indicator function but a more involved regularization term.  This more general version of the conditional gradient method is sometimes called {\em generalized conditional gradient method.}
Second, the reliance on the linear oracle~\eqref{eq.oracle} as opposed to the reliance on proximal mappings endows the conditional gradient method with the affine invariance and norm-independence properties  that play a central role in this article.

Our main developments (Theorem~\ref{thm.main} and Theorem~\ref{prop.conv.gral}) concern the behavior of the gap between the {\em primal} iterates $x_k, \; k=0,1,\dots$ generated by 
Algorithm~\ref{algo.CG} and the {\em dual} iterates $g_k:=\nabla f(x_k), \; k=0,1,\dots$ that Algorithm~\ref{algo.CG} also generates for the Fenchel  dual of~\eqref{eq.problem}, namely
\begin{equation}\label{eq.dual}
\max_u \{-f^*(u)- \Psi^*(-u)\}.
\end{equation}
Here $f^*$ and $\Psi^*$ denote the conjugates of $f$ and $\Psi$ respectively.
We will rely on the observation that for $g\in \R^n$, the set $\argmin_y\{\ip{g}{y} + \Psi(y)\}$ is precisely the subdifferential $\partial \Psi^*(-g)$.  
We will also rely on the observation that the closedness of $f$ and $\Psi$ together with Assumptions A1 and A2 imply that $\dom(\Psi)\subseteq \dom(f)$ and $\dom(f^*)\subseteq-\dom(\Psi^*)$.

\begin{algorithm}
\caption{Conditional gradient algorithm}\label{algo.CG}
\begin{algorithmic}[1]
	\STATE {\bf input:}  $(f,\Psi)$, $x_{0}\in \dom(\Psi)$
	\FOR{$k=0,1,2,\dots$}
		\STATE pick $s_{k} \in \argmin_y\{\ip{\nabla f(x_k)}{y} + \Psi(y)\}$ and  $\theta_k \in [0,1]$
		\STATE let $x_{k+1} := (1-\theta_k) x_k + \theta_k s_k$  
	\ENDFOR
\end{algorithmic}
\end{algorithm}

Starting from a current iterate $x$,  Algorithm~\ref{algo.CG} updates $x$ to $x_+:=x+\theta(s-x)$ where $s \in \partial \Psi^*(-g)$ for $g=\nabla f(x)$, and $\theta \in [0,1]$.  The choice of $g=\nabla f(x)$ and $s\in\partial \Psi^*(-g)$ yields the identity
\[
f(x) + f^*(g) + \Psi(s) + \Psi^*(-g) = \ip{g}{x-s},
\]
which can be equivalently stated as the {\em Wolfe gap:}
\begin{equation}\label{eq.wolfe.gap}
f(x)+\Psi(x) + f^*(g)+\Psi^*(-g)  = 
\ip{g}{x-s} + \Psi(x) - \Psi(s)
\end{equation}
or as 
\begin{equation}\label{eq.basic.gap}
f(s)+\Psi(s) + f^*(g)+\Psi^*(-g)  = 
f(s) -f(x) - \ip{g}{s-x} = D_f(s,x),
\end{equation}
where $D_f$ denotes the {\em Bregman distance} of $f$, that is, the function 
\[
D_f(y,x) = f(y) - f(x) - \ip{\nabla f(x)}{y-x}.
\]
In the special case when $\Psi$ is the indicator function of a closed convex set $C\subseteq \R^n$, the Wolfe gap~\eqref{eq.wolfe.gap} has the simpler and popular expression
\[
f(x)+\Psi(x) + f^*(g)+\Psi^*(-g) = \ip{g}{x-s}. 
\]

Our main developments are stated in terms of the functions $\gap$ and $\D$ described next.  
Define the duality gap function $\gap:\dom(\Psi)\times \dom(f^*) \rightarrow \R$ as follows
\begin{equation}\label{eq.gap}
\gap(x,u):=f(x)+\Psi(x) + f^*(u)+\Psi^*(-u).
\end{equation}
Define $\D:\dom(\Psi)\times \dom(\Psi) \times[0,1]\rightarrow \R$ as follows
\begin{align}\label{eq.def.D}
\D(x,s,\theta):=D_f(x+\theta(s-x),x) + \Psi(x+\theta(s-x)) - (1-\theta)\Psi(x) - \theta \Psi(s).
\end{align}

The following proposition establishes a fundamental gap reduction identity connecting $\gap$ and $\D$.  
\begin{proposition}\label{prop.gap.reduction}
Suppose $x\in \dom(\Psi), g = \nabla f(x),$ and $s\in \partial \Psi^*(-g)$  Then for $\theta\in[0,1]$ we have
\begin{equation}\label{eq.gap.reduction}
\gap(x+\theta(s-x),g) = (1-\theta)\gap(x,g) + \D(x,s,\theta).
\end{equation}
\end{proposition}
\begin{proof}
Adding~\eqref{eq.wolfe.gap} times $(1-\theta)$ and~\eqref{eq.basic.gap} times $\theta$ we get
\[
(1-\theta)(f+\Psi)(x) + \theta(f+\Psi)(s) + f^*(g)+\Psi^*(-g) = (1-\theta) \gap(x,g) + \theta D_f(s,x)
\]
Thus
\begin{align*}
&\gap(x+\theta(s-x),g) \\
&= (f+\Psi)(x+\theta(s-x)) + f^*(g)+\Psi^*(-g) \\
&= (1-\theta) \gap(x,g) + \theta D_f(s,x) +(f+\Psi)(x+\theta(s-x)) - (1-\theta)(f+\Psi)(x) - \theta(f+\Psi)(s).
\end{align*}
To finish, observe that
\begin{align*}
&\theta D_f(s,x) +(f+\Psi)(x+\theta(s-x)) - (1-\theta)(f+\Psi)(x) - \theta(f+\Psi)(s) \\
&=D_f(x+\theta(s-x),x) + \Psi(x+\theta(s-x)) - (1-\theta)\Psi(x) - \theta\Psi(s) \\
&=\D(x,s,\theta).
\end{align*}
\end{proof}
Theorem~\ref{thm.main} below shows that Algorithm~\ref{algo.CG} generates primal and dual iterates whose duality gap converges to zero provided the growth property~\eqref{eq.growth.cond} below holds and Algorithm~\ref{algo.CG} makes a judicious choice of stepsizes at each iteration.  
For a cleaner statement, Theorem~\ref{thm.main} assumes that Algorithm~\ref{algo.CG} chooses the stepsize $\theta_k$ via the following  line-search procedure 
\begin{equation}\label{eq.stepsize}
\theta_k := \argmin_{\theta\in[0,1]}\left\{(1-\theta)\gap(x_k,g_k)+ \D(x_k,s_k,\theta)\right\}.
\end{equation}
As Theorem~\ref{prop.conv.gral} shows, the main conclusions in Theorem~\ref{thm.main} continue to hold, albeit with more complicated constants, for a more flexible procedure to choose the stepsize.

The gap reduction identity~\eqref{eq.gap.reduction} in Proposition~\ref{prop.gap.reduction} implies that~\eqref{eq.stepsize} can also be equivalently stated as the {\em exact} line-search procedure 
\[
\theta_k := \argmin_{\theta\in[0,1]} (f+\Psi)(x_k+\theta(s_k-x_k)).
\]
We prefer~\eqref{eq.stepsize} because of its direct connection with the gap reduction identity~\eqref{eq.gap.reduction}.

Theorem~\ref{thm.main} hinges on the following growth property connecting the functions $\D$ and $\gap$.  
Suppose $q >1 $ and $r \in [0,1]$.  We shall say that the pair $(\D,\gap)$ satisfies the {\em $(q,r)$-growth property} if there exists a finite constant $M >0$ such that for all $x \in \dom(\Psi), \, g:=\nabla f(x),$ and $s\in\partial \Psi^*(-g)$ 
\begin{equation}\label{eq.growth.cond}
\D(x,s,\theta) \le \frac{M\theta^{q}}{q}\cdot {\gap}(x,g)^r \text{ for all } \theta \in [0,1].
\end{equation}
As we detail in Section~\ref{sec.growth.examples} below, the $(q,r)$-growth property~\eqref{eq.growth.cond} can be seen as a generalization the {\em curvature constant} defined by Jaggi~\cite{Jagg13}.
The growth property~\eqref{eq.growth.cond} also has a flavor similar to that of the {\em relative smoothness} and {\em relative continuity} properties in the context of Bregman proximal methods as discussed in~\cite{BausBT16,Lu17,LuFN18,Tebo18,VanN16,VanN17}.

The crux of Theorem~\ref{thm.main} is to combine the gap reduction identity~\eqref{eq.gap.reduction}, the growth property~\eqref{eq.growth.cond}, and the following technical lemma  due to Borwein, Li, and Yao~\cite[Lemma 4.1]{BorwLY14}.

\begin{lemma}\label{lemma.recurrence} Suppose $p>0$ and $\beta_k,\delta_k \ge 0$ are such that  
$
\beta_{k+1} \le \beta_k(1-\delta_k\beta_k^p) \text{ for } k=0,1,\dots.
$
Then
\[
\beta_k \le \left( \beta_0^{-p} +  p \cdot\sum_{i=0}^{k-1}\delta_i\right)^{-\frac{1}{p}}.
\]
\end{lemma}

Our two central results, namely Theorem~\ref{thm.main} and Theorem~\ref{prop.conv.gral}, concern the sequence $\gap_k,\; k=0,1,\dots$ of {\em best duality gaps} defined as follows.  For $k=0,1,\dots$ let
\begin{equation}\label{eq.gap.k}
\gap_k := \min_{i=0,1,\dots,k} \gap(x_k,g_i)
\end{equation}
where $g_i=\nabla f(x_i)$ for $i=0,1,\dots$.  

Observe that $\gap_k = \gap(x_k,g^\text{best}_k)$ 
where $g^\text{best}_k$ is the best  of the dual iterates $g_i, \; i=0,1,\dots,k$.  A simple calculation shows that both $\gap_k$ and $g^\text{best}_k$ can be updated as follows: $\gap_0 = \gap(x_0,g_0)$ and for $k=0,1,\dots$
\[
\gap_{k+1} = \min\{\gap(x_{k+1},g_{k+1}),\gap_k + (f+\Psi)(x_{k+1}) - (f+\Psi)(x_k) \},
\]
and
\[
g^\text{best}_{k+1} = \left\{ \begin{array}{rl} g_{k+1} & \text{ if } \gap(x_{k+1},g_{k+1})<\gap_k + (f+\Psi)(x_{k+1}) - (f+\Psi)(x_k) \\
g^\text{best}_{k} & \text{ otherwise.}
\end{array}\right.
\]

\begin{theorem}\label{thm.main} Suppose Algorithm~\ref{algo.CG} chooses $\theta_k\in [0,1]$ via~\eqref{eq.stepsize} and  $q > 1$ and  $r\in [0,1]$ are such that
$(\D,\gap)$ satisfy the $(q,r)$-growth property~\eqref{eq.growth.cond} for some finite $M>0$. 
Then for $k=0,1,\dots$
\begin{equation}\label{eq.recurrence}
\gap_{k+1} \le \gap_k\left(1- \frac{q-1}{q} \cdot \min\left\{1,\left(\frac{\gap_k^{1-r}}{M}\right)^\frac{1}{q-1} \right\}\right).
\end{equation}
When $r=1$ we have linear convergence
\begin{equation}\label{eq.linear}
\gap_k\le\gap_0\left(1-\frac{q-1}{q}\cdot\min\left\{1,\frac{1}{M^{\frac{1}{q-1}}}\right\} \right)^k.
\end{equation}
When $r\in[0,1)$ we have an initial linear convergence regime
\begin{equation}\label{eq.linear.initial}
\gap_{k} \le \gap_0\left(1-\frac{q-1}{q} \right)^{k}, \; k=0,1,2,\dots,k_0
\end{equation}
where $k_0$ is the smallest $k$ such that 
$\gap_k^{1-r} \le M$.  Then for $k\ge k_0$ we have a sublinear convergence regime
\begin{equation}\label{eq.sublinear}
\gap_k\le\left(\gap_{k_0}^\frac{r-1}{q-1}+\frac{1-r}{q}\cdot\frac{1}{M^{\frac{1}{q-1}}} \cdot (k-k_0)\right)^\frac{q-1}{r-1}.
\end{equation}

\end{theorem}
\begin{proof}
Identity~\eqref{eq.gap.reduction} in Proposition~\ref{prop.gap.reduction}  implies that
\[
\gap(x_{k+1},g_k) = (1- \theta_k) \gap(x_k,g_k) + \D(x_k,s_k,\theta_k).
\]
Thus~\eqref{eq.gap.k}, the stepsize choice~\eqref{eq.stepsize}, and growth property~\eqref{eq.growth.cond}  imply that
\begin{align*}
\gap_{k+1} &\le \gap_k - \theta_k \cdot\gap(x_k,g_k) + \D(x_k,s_k,\theta_k)
\\&=\gap_k + \min_{\theta\in[0,1]}\left\{-\theta\cdot\gap(x_k,g_k)+\D(x_k,g_k,\theta)  \right\}
\\&\le\gap_k+\min_{\theta\in[0,1]}\left\{-\theta\cdot\gap(x_k,g_k)+\frac{M\theta^{q}}{q}\cdot \gap(x_k,g_k)^r  \right\}
\\&\le
\gap_k\left(1- \frac{q-1}{q} \cdot \min\left\{1,\left(\frac{\gap_k^{1-r}}{M}\right)^\frac{1}{q-1} \right\}\right).
\end{align*}
The last step above follows because the minimum in the second to last step is attained at
\[
\hat \theta = \min\left\{1,\left(\frac{\gap(x_k,g_k)^{1-r}}{M}\right)^\frac{1}{q-1}\right\}
 \ge \min\left\{1,\left(\frac{\gap_k^{1-r}}{M}\right)^\frac{1}{q-1}\right\}.
\]
Thus $M\hat \theta^{q} \cdot \gap(x_k,g_k)^r \le \hat\theta  \cdot \gap(x_k,g_k)$ and consequently
\begin{align*}
\gap_k& +\min_{\theta\in[0,1]}\left\{-\theta\cdot\gap(x_k,g_k)+\frac{M\theta^{q}}{q}\cdot \gap(x_k,g_k)^r  \right\} \\
&\le \gap_k - \hat \theta \cdot \gap(x_k,g_k) + \frac{1}{q}\cdot\hat \theta \cdot \gap(x_k,g_k) \\
&\le \gap_k\left(1- \frac{q-1}{q} \cdot\hat \theta \right)
\\
&\le \gap_k\left(1- \frac{q-1}{q}\cdot\min\left\{1,\left(\frac{\gap_k^{1-r}}{M}\right)^\frac{1}{q-1}\right\}\right).
\end{align*}

Therefore~\eqref{eq.recurrence} is established. 
Inequality~\eqref{eq.recurrence}  automatically implies~\eqref{eq.linear} when $r=1$ and also~\eqref{eq.linear.initial} for $k\le k_0$ when $r\in [0,1)$.  When $r\in[0,1)$ and $k\ge k_0$ inequality~\eqref{eq.sublinear} follows from~\eqref{eq.recurrence} and Lemma~\ref{lemma.recurrence} applied to $p:=\frac{1-r}{q-1}$ and $\beta_k:=\gap_k, \delta_k:=\frac{q-1}{q}\cdot\frac{1}{M^{\frac{1}{q-1}}}$ for $k=k_0,k_0+1,\dots$. 
\end{proof}

Theorem~\ref{thm.main} shows that for fixed $q > 1$, the value of $r\in[0,1]$ in the $(q,r)$-growth property~\eqref{eq.growth.cond} yields a convergence rate for $\gap_k\rightarrow 0$ somewhere in the spectrum from the sublinear rate $\Oh(1/k^{q-1})$ when $r=0$ to a linear rate when $r=1$.  The convergence rate increases as $r$ increases to $1$.  For fixed $r\in (0,1)$ the converge rate also increases as the parameter $q>1$ in the $(q,r)$-growth property~\eqref{eq.growth.cond} increases. Section~\ref{sec.growth.examples} below elaborates on these observations and other  interesting consequences of Theorem~\ref{thm.main}.

We next describe a more flexible alternative to~\eqref{eq.stepsize} to choose the stepsize $\theta_k$.  Suppose $c,\rho \in (0,1)$.  A simple backtracking procedure can easily choose $\theta_k\in[0,1]$ such that $\rho\cdot \hat \theta_k\le \theta_k \le \hat \theta_k$ for  
\begin{equation}\label{eq.stepsize.relax}
\hat \theta_k := \max\left\{\theta\in[0,1]:(1-\theta)\gap(x_k,g_k)+ \D(x_k,s_k,\theta) \le (1-c\cdot \theta)\gap(x_k,g_k)\right\}.
\end{equation}
The inequality in~\eqref{eq.stepsize.relax} can be seen as an {\em Armijo rule} condition for the function $\varphi:[0,1]\rightarrow \R$ defined as $\varphi(\theta):= (1-\theta)\gap(x_k,g_k)+ \D(x_k,s_k,\theta)$.  Indeed, a simple calculation shows that $\varphi(0) = \gap(x_k,g_k)$ and the right derivative of $\varphi$ at zero satisfies
$$\varphi'_+(0) \le -\gap(x_k,g_k)<0.$$ Therefore the inequality in~\eqref{eq.stepsize.relax} can be rewritten in the usual Armijo rule format~\cite[Chapter 3]{NoceW06}:
\[
\varphi(\theta) \le \varphi(0) + \tilde c \cdot \theta \cdot\varphi'_+(0).
\]
for some $\tilde c \le c$.
The Armijo rule condition~\eqref{eq.stepsize.relax} is evidently easier to satisfy, and therefore more flexible, than the exact line-search~\eqref{eq.stepsize} which corresponds to choosing $\theta_k:=\argmin_{\theta \in [0,1]} \varphi(\theta)$.

\medskip

A straightforward modification of the proof of Theorem~\ref{thm.main} yields the following result.

\begin{theorem}\label{prop.conv.gral}
Suppose $c, \rho\in (0,1)$ are such that $c+\rho > 1$ and Algorithm~\ref{algo.CG} chooses $\theta_k\in [0,1]$ so that $\rho\cdot \hat \theta_k \le \theta_k \le \hat \theta_k$ for $\hat \theta_k$ as in~\eqref{eq.stepsize.relax}. 
Suppose $q > 1$ and  $r\in [0,1]$ are such that
 $(\D,\gap)$ satisfy the $(q,r)$-growth property~\eqref{eq.growth.cond} for some finite $M>0$.
Then for $k=0,1,\dots$
\begin{equation}\label{eq.recurrence.gral}
\gap_{k+1} \le 
\gap_k\left(1-(c+\rho -1)\cdot\min\left\{1,\left(\frac{q(1-c)}{M}\cdot\gap_k^{1-r}\right)^\frac{1}{q-1}\right\}\right). 
\end{equation}
When $r=1$ we have linear convergence
\begin{equation}\label{eq.linear.gral}
\gap_k\le\gap_0\left(1-(c+\rho -1)\cdot 
\min\left\{1,
\left(\frac{ q(1-c)}{M}\right)^\frac{1}{q-1}\right\} \right)^k.
\end{equation}
When $r\in[0,1)$ we have an initial linear convergence regime
\begin{equation}\label{eq.linear.gral.initial}
\gap_{k} \le \gap_0\left(1-(c+\rho-1)\right)^{k}, \; k=0,1,2,\dots,k_0
\end{equation}
where $k_0$ is the smallest $k$ such that $\gap_k^{1-r} \le \frac{M}{q(1-c)}$.  Subsequently for $k\ge k_0$ we have a sublinear convergence regime
\begin{equation}\label{eq.sublinear.gral}
\gap_k\le\left(\gap_{k_0}^\frac{r-1}{q-1}+\frac{(1-r)(c+\rho -1)}{q-1}\cdot\left(\frac{q(1-c)}{ M}\right)^\frac{1}{q-1}\cdot (k-k_0) \right)^\frac{q-1}{r-1}.
\end{equation}
\end{theorem} 
\begin{proof}
As noted above, this is a modification of the proof of Theorem~\ref{thm.main}.
Identity~\eqref{eq.gap.reduction} in Proposition~\ref{prop.gap.reduction} yields
\[
\gap(x_{k+1},g_k) = (1-\theta_k)\gap(x_k,g_k) + \D(x_k,s_k,\theta_k).  
\]
Hence~\eqref{eq.gap.k} implies that
\[
\gap_{k+1} \le \gap_k - \theta_k\cdot \gap(x_k,g_k) + \D(x_k,s_k,\theta_k). 
\]
Since $\rho\cdot\hat \theta_k \le \theta_k \le \hat \theta_k$ for $\hat \theta_k$ as in~\eqref{eq.stepsize.relax}, it follows that
\begin{align}
\gap_{k+1} & \le \gap_k -\rho \cdot\hat \theta_k \cdot\gap(x_k,g_k) + \D(x_k,s_k,\theta_k)  \notag\\
&\le \gap_k - \rho \cdot\hat\theta_k\cdot \gap(x_k,g_k) + (1-c)\cdot\theta_k\cdot\gap(x_k,g_k) \notag\\
&\le  \gap_k - (c+\rho -1)\cdot\hat\theta_k\cdot\gap(x_k,g_k) \notag\\
& \le \gap_k(1 - (c+\rho-1)\cdot\hat \theta_k). \label{eq.gral.step1}
\end{align}
The growth property~\eqref{eq.growth.cond} and~\eqref{eq.gap.k} imply that
the minimizer $\hat \theta_k$ in~\eqref{eq.stepsize.relax} satisfies
\begin{align}
\hat \theta_k &= \max\left\{\theta\in[0,1]: \D(x_k,s_k,\theta) \le (1-c)\cdot\theta\cdot\gap(x_k,g_k)\right\} 
\notag\\
&\ge \max\left\{\theta\in[0,1]: \frac{M \theta^q}{q}\gap(x_k,g_k)^r \le (1-c)\cdot\theta\cdot\gap(x_k,g_k)\right\} 
\notag\\
&= \min\left\{1,\left(\frac{q(1-c)}{M}\cdot\gap(x_k,g_k)^{1-r}\right)^\frac{1}{q-1}\right\}\notag\\
&\ge \min\left\{1,\left(\frac{q(1-c)}{M}\cdot\gap_k^{1-r}\right)^\frac{1}{q-1}\right\}.\label{eq.gral.step2}
\end{align}
Inequality~\eqref{eq.recurrence.gral} follows by putting together~\eqref{eq.gral.step1} and~\eqref{eq.gral.step2}.
Inequality~\eqref{eq.recurrence.gral}  automatically implies~\eqref{eq.linear.gral} when $r=1$ and also~\eqref{eq.linear.gral.initial} for $k\le k_0$ when $r\in [0,1)$.  When $r\in[0,1)$ and $k\ge k_0$ inequality~\eqref{eq.sublinear.gral} follows from~\eqref{eq.recurrence.gral} and Lemma~\ref{lemma.recurrence} applied to $p:=\frac{1-r}{q-1}$ and $\beta_k:=\gap_k, \delta_k:=(c+\rho-1)\left(\frac{q(1-c)}{M}\right)^{\frac{1}{q-1}}$ for $k=k_0,k_0+1,\dots$. 
\end{proof}

We note that expression~\eqref{eq.recurrence.gral} in Theorem~\ref{prop.conv.gral} is identical to expression~\eqref{eq.recurrence} in Theorem~\ref{thm.main} for the ideal choice $\rho = 1$ and $c = (q-1)/q$.  Likewise for~\eqref{eq.linear.gral},~\eqref{eq.sublinear.gral} and~\eqref{eq.linear},~\eqref{eq.sublinear}.

\bigskip

To conclude this section, we next discuss the norm-independence and affine invariance of our developments.  The growth property~\eqref{eq.growth.cond} as well as Theorem~\ref{thm.main} and Theorem~\ref{prop.conv.gral} are automatically norm-independent since they do not rely on any norms.  They are also affine invariant as we next detail.

Consider an affine {\em reparametrization} of $\R^n$ of the form $x:=A\tilde x + b$ where $A:\R^n\rightarrow \R^n$ is a linear bijection.  Let $\tilde f, \tilde \Psi:\R^n\rightarrow \R\cup\{\infty\}$ be defined as the corresponding reparametrizations of $f$ and $\Psi$, namely
\[
\tilde f(\tilde x):= f(A\tilde x + b) \; \text{ and } \; \tilde \Psi(\tilde x):= \Psi(A\tilde x + b).
\]
Let $A^*:\R^n\rightarrow \R^n$ denote the adjoint of $A$.  Some straightforward calculations show that if $x = A\tilde x+b$ then
\[
\nabla \tilde f(\tilde x) = A^* \nabla f(x) \text{  and  } \tilde f^*(A^*u) = f^*(u) - \ip{b}{u} \text{ for all } u\in \R^n,
\]
and likewise for $\Psi$ in lieu of $f$.  Consequently, if $x = A\tilde x+b$ then
\[
\partial \Psi^*(-A^*\nabla f(x)) = A\partial \tilde \Psi^*(-\nabla \tilde f(\tilde x)) + b.
\]
Hence if $x = A\tilde x + b$ then for $g:=\nabla f(x)$ and  $\tilde g := \nabla \tilde f(\tilde x)$ we have $\tilde g = A^*g$ and $\partial \Psi^*(-g) = A\partial \Psi^*(-\tilde g)+b$.  Thus we can assume that the oracle in Assumption A2 
yield respectively $s\in\partial \Psi^*(-g)$ and $\tilde s\in\partial \tilde \Psi^*(-\tilde g)$ that satisfy $s = A\tilde s+b$.
The identities $x=A\tilde x+b, \; \tilde g = A^*g, \; s = A\tilde s+b$ in turn imply that
\[
\gap(x,g) = \widetilde{\gap}(\tilde x,\tilde g)\]
 and 
 \[ \D(x,s,\theta) = \widetilde{\D}(\tilde x,\tilde s,\theta) \text{ for all } \theta \in [0,1].
\]  
It thus follows that the growth property~\eqref{eq.growth.cond}, stepsize selection procedures~\eqref{eq.stepsize} and~\eqref{eq.stepsize.relax}, as well as Theorem~\ref{thm.main} and Theorem~\ref{prop.conv.gral} are all affine invariant.

\section{Growth  property in some special cases}
\label{sec.growth.examples}
The subsections below detail some straightforward consequences of Theorem~\ref{thm.main} when the growth property holds for some special values of $(q,r)$.  These special cases recover and sharpen some norm-dependent convergence results previously developed  in the state-of-the-art literature on the conditional gradient method~\cite{Kerd21,Nest18,Ghad19}.   To make the latter connection more precise and explicit, Proposition~\ref{prop.suff.r=0}, Proposition~\ref{prop.suff.2.1}, and Proposition~\ref{prop.suff.gral} below show that the uniform smoothness and uniform convexity assumptions previously made in each of~\cite{Kerd21,Nest18,Ghad19} imply the $(q,r)$ growth property~\eqref{eq.growth.cond} for various values of $(q,r)$.

We next recall the concepts of uniform smoothness and uniform convexity.
A detailed discussion on these concepts is presented in~\cite{KerdDP21}.  Both  uniform smoothness and uniform convexity are stated and formalized in terms of some norm.  Thus throughout this section we assume that $\R^n$ is endowed with a norm $\|\cdot\|$ and let $\|\cdot\|^*$ denote its dual norm.  

A convex function $f:\R^n \rightarrow \R\cup\{\infty\}$ is {\em uniformly smooth} on $C\subseteq \dom(f)$ with respect to $\|\cdot\|$ if there exist constants $q\in(1,2]$ and $L>0$ such that for $x,y\in C$ 
\begin{equation}\label{eq.unif.smooth}
f(x+\theta(y-x)) \ge (1-\theta)f(x) + \theta f(y) - \frac{L}{q} \theta(1-\theta) \|y-x\|^q \text{ for all } \theta \in [0,1]. 
\end{equation}
It is customary to say that $f$ is $L$-smooth on $C$ when \eqref{eq.unif.smooth} holds with $q=2$.  In particular, {\em uniform smoothness} is a relaxation of {\em smoothness.}
It is easy to see that~\eqref{eq.unif.smooth} holds for $q=\nu+1$ when $\nabla f$ is $\nu$-H\"older continuous on $C$ with H\"older constant $L$.   In particular, $f$ is $L$-smooth on $C$ when $\nabla f$ is Lipschitz continuous on $C$ with Lipschitz constant $L$.

As detailed in~\cite{KerdDP21}, when $f$ is differentiable, property~\eqref{eq.unif.smooth} implies that for all $ x,y\in C$
\begin{equation}\label{eq.unif.smooth.implied}
D_f(y,x) \le \frac{L}{q}\|y-x\|^q. 
\end{equation}
A function $\Psi:\R^n \rightarrow \R\cup\{\infty\}$ is {\em uniformly convex} on $C\subseteq \dom(\Psi)$ with respect to $\|\cdot\|$  if there exist constants $p\ge 2$ and $\mu>0$ such that for all $x,y\in C$ and $\theta \in [0,1]$
\begin{equation}\label{eq.unif.convex}
\Psi(x+\theta(y-x)) \le (1-\theta)\Psi(x) + \theta \Psi(y) - \frac{\mu}{p} \theta(1-\theta) \|y-x\|^p \text{ for all } \theta \in [0,1]. 
\end{equation}
It is customary to say that $\Psi$ is strongly convex on $C$ when ~\eqref{eq.unif.convex} holds for $p=2$.  In particular, {\em uniform convexity} is a relaxation of {\em strong convexity.}

Again, as detailed in~\cite{KerdDP21}, property~\eqref{eq.unif.convex} implies that for all $x,y\in C$ and $g\in \partial \Psi(y)$ 
\[
\Psi(x) \ge \Psi(y) + \ip{g}{s-y} + \frac{\mu}{p}\|x-y\|^p.
\]
In particular,  property~\eqref{eq.unif.convex} implies that if $C = \dom(\Psi)$ and $s = \argmin_y \Psi(y)$ then for all $x\in \dom(\Psi)$
\begin{equation}\label{eq.unif.convex.implied}
\Psi(x) - \Psi(s) \ge \frac{\mu}{p}\|x-s\|^p.
\end{equation}

A closed convex set $C\subseteq\R^n$ is {\em uniformly convex} if there exist constants $p\ge 2$ and $\mu>0$ such that for all $x,y\in C,$  $\theta \in [0,1]$, and $z\in \R^n$ with $\|z\|\le 1$
\begin{equation}\label{eq.unif.convex.set}
x+\theta(y-x) + \frac{\mu}{p}\theta(1-\theta)\|y-x\|^p z \in C.
\end{equation}
It is customary to say that $C$ is strongly convex when~\eqref{eq.unif.convex.set} holds for $p=2.$

Similar to the above implications~\eqref{eq.unif.smooth} $\Rightarrow$~\eqref{eq.unif.smooth.implied} and~\eqref{eq.unif.convex} $\Rightarrow$~\eqref{eq.unif.convex.implied}, property~\eqref{eq.unif.convex.set} implies that if $g\in \R^n$ and $s = \argmin_{y\in C} \ip{g}{y}$ then for all $x\in C$ 
\begin{equation}\label{eq.unif.convex.set.implied}
\ip{g}{x-s} \ge \frac{\mu}{p}  \|g\|^* \cdot \|x-s\|^p.
\end{equation}

\subsection{The case $q>1,\; r=0$: sublinear convergence}
Suppose the growth property~\eqref{eq.growth.cond} holds for $q>1$ and $r =0,$ that is, for $x\in \dom(\Psi), s \in \partial \Psi^*(-\nabla f(x))$
\begin{equation}\label{eq.r=0}
\D(x,s,\theta) \le \frac{M\theta^{q}}{q} \text{ for all } \theta \in [0,1].
\end{equation}
In this case $\gap_1\le M/q\le M$ and thus Theorem~\ref{thm.main} implies that
\begin{align}\label{eq.conv.0}
\gap_k &\le
\left(\gap_1^{-\frac{1}{q-1}} + \frac{1}{q}\cdot \frac{k-1}{M^\frac{1}{q-1}}\right)^{-(q-1)} \notag\\
&\le
\left(
\frac{q^\frac{1}{q-1}}{M^\frac{1}{q-1}} 
+ \frac{1}{q}\cdot \frac{k-1}{M^\frac{1}{q-1}}
\right)^{-(q-1)} \notag\\
&= M\left(\frac{q}{k-1+q^\frac{q}{q-1}}\right)^{q-1}.
\end{align}
The case $(q,r) = (2,0)$ is particularly interesting.  In this case we get
\begin{equation}\label{eq.classic}
\gap_k \le \frac{2M}{k+3}.
\end{equation}
This bound applies in the especially important and popular case when $\Psi$ is the indicator function $\delta_C$ of some compact convex set $C\subseteq \R^n$ as discussed in~\cite{Jagg13,FreuG16}. In this case for $x\in \dom(\Psi), s \in \partial \Psi^*(-\nabla f(x))$
 we have
 \[
 \D(x,s,\theta) = D_f(x+\theta(s-x),x) \text{ for all } \theta\in [0,1]
 \]
and thus~\eqref{eq.r=0} holds with $q=2$ and $M$ equal to the following {\em curvature constant} defined by Jaggi~\cite{Jagg13}, provided that it is finite:
\begin{equation}\label{eq.jaggi}
\mathcal C_f := \sup_{x,s\in C\atop \theta\in (0,1]}\frac{2 D_f(x+\theta(s-x),x)}{\theta^2}.
\end{equation}
Therefore~\eqref{eq.classic} recovers the iconic 
convergence rate of the conditional gradient method established in~\cite{Jagg13}.  However, Theorem~\ref{thm.main} reveals that the convergence could be faster as~\eqref{eq.r=0} may hold for $M\le \mathcal C_{f}$.

\medskip

The following proposition describes an important class of problems that satisfies~\eqref{eq.r=0}.  We rely on the following norm-dependent concept. Suppose $\|\cdot\|$ is a norm in $\R^n$ and $C\subseteq \R^n$ is nonempty and bounded.  The {\em diameter} of $C$ is
\[
\diam(C):=\sup_{x,y\in C} \|x-y\|.
\]
\begin{proposition}\label{prop.suff.r=0} The $(q,0)$-growth property~\eqref{eq.r=0} holds if $f$ is $q$-uniformly smooth and $\dom(\Psi)$ is bounded.  More precisely, the $(q,0)$-growth property~\eqref{eq.r=0} holds for $M=L\cdot\diam(\dom(\Psi))^q$ if there is a norm $\|\cdot\|$ in $\R^n$ such that
\eqref{eq.unif.smooth} holds for $x,y\in \dom(\Psi)$.
\end{proposition}
\begin{proof}
Suppose~\eqref{eq.unif.smooth} holds for some constants $q>1$ and $L >0$. Then~\eqref{eq.def.D} and~\eqref{eq.unif.smooth.implied} imply that for all $x,s\in \dom(\Psi)$ and $\theta\in [0,1]$
\[
\D(x,s,\theta) \le D_f(x+\theta(s-x),x) \le \frac{L }{q} \|\theta(s-x)\|^q = \frac{L\|s-x\|^q }{q}\cdot \theta^q.
\] 
Thus~\eqref{eq.r=0} holds for $M=L\cdot\diam(\dom(\Psi))^q$.
\end{proof}

Proposition~\ref{prop.suff.r=0} and~\eqref{eq.conv.0} readily yield the following norm-dependent convergence result for a class of problems previously considered in~\cite{Ghad19,Nest18}. Suppose $\dom(\Psi)$ is bounded and $\nabla f$ is $\nu$-H\"older continuous on $\dom(\Psi)$ for some $\nu > 0$.  In this case $f$ is $(\nu+1)$-uniformly smooth on $\dom(\Psi)$. Thus Proposition~\ref{prop.suff.r=0} implies that the $(q,0)$-growth property~\eqref{eq.r=0} holds for $q = \nu+1$ and $M = L \cdot \diam(\dom(\Psi))^{1+\nu}$ where $L$ is the H\"older continuity constant of $\nabla f$.  Therefore~\eqref{eq.conv.0} implies that
\begin{align}\label{eq.nu.conv}
\gap_k &\le L \cdot \diam(\dom(\Psi))^{1+\nu}\left(\frac{\nu+1}{k-1+(\nu+1)^\frac{\nu+1}{\nu}}\right)^{\nu} \notag \\
&\le L \cdot \diam(\dom(\Psi))^{1+\nu}\left(\frac{\nu+1}{k}\right)^{\nu}.
\end{align}
Therefore we recover the $\Oh(1/k^\nu)$ convergence results shown in~\cite[Corollary 1]{Nest18} and in~\cite[Corollary 1]{Ghad19} with an improvement on the latter as~\eqref{eq.nu.conv} does not include the additional log factor in~\cite[Corollary 1 and Equation (2.30)]{Ghad19}.  The constant in front of $1/k^\nu$ in~\eqref{eq.nu.conv} is simpler and sharper than the constants in~\cite[Corollary 1]{Nest18} and in~\cite[Equation (2.30)]{Ghad19}.  Furthermore, Theorem~\ref{thm.main} reveals that the convergence is generally faster 
as Proposition~\ref{prop.suff.r=0} implies that~\eqref{eq.r=0} and~\eqref{eq.conv.0} will typically hold for $M\le L \cdot \diam(\dom(\Psi))^{1+\nu}$.

\subsection{The case $(q,r)= (2,1)$: linear convergence}
\label{sec.2.1}
Suppose the growth property~\eqref{eq.growth.cond} holds for $(q,r) = (2,1),$ that is, for $x\in \dom(\Psi), s \in \partial \Psi^*(-\nabla f(x))$
\begin{equation}\label{eq.q=2.r=1}
\D(x,s,\theta) \le \frac{M\theta^{2}}{2}\cdot {\gap}(x,g) \text{ for all } \theta \in [0,1].
\end{equation}
In this case Theorem~\ref{thm.main} readily implies that 
\begin{equation}\label{eq.conv.2.1}
\gap_k \le \gap_0\cdot
\left(1-\frac{1}{2}\cdot\min\left\{1,\frac{1}{M}\right\}\right)^{k}. 
\end{equation}

Proposition~\ref{prop.suff.2.1} below describes some important classes of problems that satisfies~\eqref{eq.q=2.r=1}. We omit the proof of Proposition~\ref{prop.suff.2.1} as it is a special case of the more general Proposition~\ref{prop.suff.gral} in Section~\ref{sec.other.q.r} below.

\begin{proposition}\label{prop.suff.2.1}
Suppose $f$ is $2$-uniformly smooth on $\dom(\Psi)$.
Then the $(2,1)$-growth property~\eqref{eq.q=2.r=1} holds if $\Psi$ is $2$-uniformly convex on $\dom(\Psi)$, or if $\Psi = \delta_C$ for a $2$-uniformly convex and closed set $C\subseteq \R^n$ and $\nabla f$ is bounded away from zero in $C$.  More precisely:
\begin{itemize}
\item[(a)] The $(2,1)$-growth property~\eqref{eq.q=2.r=1} holds for $M = 2L/\mu$ if there is a norm $\|\cdot\|$ in $\R^n$ such that both~\eqref{eq.unif.smooth} and~\eqref{eq.unif.convex} hold for $p=q=2$ and $x,y\in \dom(\Psi)$.
\item[(b)] The $(2,1)$-growth property~\eqref{eq.q=2.r=1} holds for $M = 2L/(\ell\mu)$ if $\Psi = \delta_C$ for a closed convex set $C\subseteq \R^n$, there is a norm $\|\cdot\|$ in $\R^n$ such that both~\eqref{eq.unif.smooth} and~\eqref{eq.unif.convex.set} hold for $p=q=2$ and $x,y\in C$, and 
\[
\ell := \inf_{x\in C} \|\nabla f(x)\|^* > 0.
\]
\end{itemize}
\end{proposition}

Proposition~\ref{prop.suff.2.1} and~\eqref{eq.conv.2.1} readily yield the following norm-dependent linear convergence results for two classes of problems previously considered in~\cite{Ghad19,KerdL21,LeviP66}.

First, if $\nabla f$ is Lipschitz continuous on $\dom(\Psi)$ and $\Psi$ is strongly convex then Proposition~\ref{prop.suff.2.1}(a) implies that the $(2,1)$-growth property~\eqref{eq.q=2.r=1} holds with $M = L/\mu$ where $L$ and $\mu$ are respectively the Lipschitz continuity constant of $\nabla f$ and the strong convexity constant of $\Psi$. Therefore~\eqref{eq.conv.2.1} implies that
\[
\gap_k \le \gap_0 \cdot\left(1-
\frac{1}{2}\cdot\min\left\{1,\frac{\mu}{2L}\right\}
\right)^{k}. 
\]
Hence we recover the linear convergence result shown in~\cite[Corollary 1 and Equation (2.32)]{Ghad19} but with a simpler and sharper expression.  Furthermore, the convergence is generally faster as~\eqref{eq.q=2.r=1} and~\eqref{eq.conv.2.1} will typically hold for $M\le 2L/\mu$.

Second, if $\Psi=\delta_C$, $\nabla f$ is Lipschitz continuous and bounded away from zero on $C$, and $C\subseteq\R^n$ is a closed and strongly set then Proposition~\ref{prop.suff.2.1}(b) implies that the $(2,1)$-growth property~\eqref{eq.q=2.r=1} holds with $M = 2L/(\ell\mu)$ where $L,\mu,$ and $\ell$ are respectively the Lipschitz continuity constant of $\nabla f$, the strong convexity constant of $C$ and the lower bound on $\|\nabla f(x)\|^*$ on $C$.  Therefore~\eqref{eq.conv.2.1} implies that
\[
\gap_k \le \gap_0 \cdot\left(1-
\frac{1}{2}\cdot\min\left\{1,\frac{\ell\mu}{2L}\right\}
\right)^{k}. 
\]
Hence we recover the linear convergence result shown in~\cite[Theorem 6.1(5)]{LeviP66} and in~\cite[Theorem 2.2]{KerdL21} with a slightly improved convergence rate.  Theorem~\ref{thm.main} shows that the convergence rate is generally better as~\eqref{eq.q=2.r=1} and~\eqref{eq.conv.2.1} typically hold for $M\le 2L/(\ell\mu)$.

\subsection{Other values $(q,r)$: sublinear to linear convergence}
\label{sec.other.q.r}
Suppose the growth property~\eqref{eq.growth.cond} holds for some $q > 1,\; r\in(0,1),$ and $M >0$.  We next discuss the spectrum of convergence rates from sublinear to linear that Theorem~\ref{thm.main} yields depending on the values of $(q,r)$.  To simplify notation and exposition, we shall assume that $\gap_0 \le M$ throughout this subsection so that we skip the initial linear convergence regime~\eqref{eq.linear.initial} in Theorem~\ref{thm.main}.  This assumption can be relaxed either at the expense of some notational overhead to deal with the initial linear convergence regime~\eqref{eq.linear.initial}, or at the expense of artificially increasing the constant $M$.

Suppose $\gap_0 \le M$. Then for $r \in [0,1)$ Theorem~\ref{thm.main} implies that 
\begin{equation}\label{eq.conv.gral}
\gap_k  \le\left(\frac{1}{M^\frac{1-r}{q-1}}+\frac{1-r}{q}\cdot\frac{k}{M^{\frac{1}{q-1}}}\right)^\frac{q-1}{r-1}
= M\left(1+\frac{q-1}{q \cdot M^{\frac{r}{q-1}}}\cdot k \cdot \frac{1-r}{q-1}\right)^{-\frac{q-1}{1-r}}.
\end{equation}
For $q > 1$ fixed, the right-most expression in~\eqref{eq.conv.gral} decreases monotonically with $r$ from the sublinear rate 
$
M \left(\frac{q}{k+q}\right)^{q-1}
$
for $r=0$
to the following asymptotic linear rate for $r\uparrow 1$
\begin{equation}\label{eq.linear.conv}
M e^{-\rho k} = M \left(e^{-\rho}\right)^k,
\end{equation}
where $$\rho := \frac{q-1}{q\cdot M^{\frac{1}{q-1}}} \in (0,1).$$
Since $e^{-\rho} \approx 1-\rho$ for small $\rho>0$, it follows that the asymptotic linear rate~\eqref{eq.linear.conv} is approximately the same linear rate given by Theorem~\ref{thm.main} when $r = 1$, namely
\[
\gap_k \le \gap_0 \left(1-\frac{q-1}{q\cdot M^{\frac{1}{q-1}}} \right)^k.
\]

The following general version of Proposition~\ref{prop.suff.2.1} describes a broad class of problems that satisfies the $(q,r)$-growth property for various values of $(q,r)$ depending on the degree of uniform smoothness and uniform convexity of $f$ and $\Psi$.

\begin{proposition}\label{prop.suff.gral}
Suppose $f$ is $q$-uniformly smooth on $\dom(\Psi)$.
Then the $(q,r)$-growth property~\eqref{eq.growth.cond} holds for $r = q/p$ if $\Psi$ is $p$-uniformly convex on $\dom(\Psi)$, or if $\Psi = \delta_C$ for a $p$-uniformly convex and closed set $C\subseteq \R^n$ and $\nabla f$ is bounded away from zero in $C$.  More precisely:
\begin{itemize}
\item[(a)] The $(q,r)$-growth property~\eqref{eq.growth.cond} holds for $r = q/p$ and $M = L\cdot \left(p/\mu\right)^{\frac{q}{p}}$ if there is a norm $\|\cdot\|$ in $\R^n$ such that both~\eqref{eq.unif.smooth} and~\eqref{eq.unif.convex} hold for $x,y\in \dom(\Psi)$.
\item[(b)] The $(q,r)$-growth property~\eqref{eq.growth.cond} holds for $r = q/p$ and $M = L\cdot \left(p/(\ell\mu)\right)^{\frac{q}{p}}$ if $\Psi=\delta_C$ for 
a closed convex set $C\subseteq \R^n$, and there is a norm $\|\cdot\|$ in $\R^n$ such that both~\eqref{eq.unif.smooth} and~\eqref{eq.unif.convex.set} hold for $x,y\in C$, and 
\[
\ell := \inf_{x\in C} \|\nabla f(x)\|^* > 0.
\]
\end{itemize}
\end{proposition}
\begin{proof}
\begin{itemize}
\item[(a)]
Suppose~\eqref{eq.unif.smooth} holds for some constants $q>1$ and $L >0$, and also~\eqref{eq.unif.convex} holds  for some constants $p>1$ and $\mu >0$.  Then for all $x\in \dom(\Psi)$ and $g :=\nabla f(x)$ the tilted function $\Psi_g := \Psi + \ip{g}{\cdot}$ is also $p$-uniformly convex on $\dom(\Psi_g) = \dom(\Psi)$ for the same $p$ and $\mu$.  Consequently~\eqref{eq.unif.convex.implied} implies that $s := \argmin_y\{\Psi(y) + \ip{g}{y}\} = \argmin_y \Psi_g(y)$ satisfies
\[
\gap(x,g) = \Psi(x) + \ip{g}{x} - \Psi(s) - \ip{g}{s} = \Psi_g(x) - \Psi_g(s) \ge \frac{\mu}{p}\|x-s\|^p.
\]
On the other hand,~\eqref{eq.def.D} and~\eqref{eq.unif.smooth.implied} imply that
\[
\D(x,s,\theta) \le D_f(x+\theta(s-x),x) \le \frac{L \theta^q}{q} \|s-x\|^q.
\] 
It thus follows that~\eqref{eq.growth.cond} holds for $q$ and $r=q/p$ with $M = L\cdot \left({p}/{\mu}\right)^{\frac{q}{p}}$.
\item[(b)] Suppose~\eqref{eq.unif.smooth} holds for some constants $q>1$ and $L >0$,~\eqref{eq.unif.convex.set} holds  for some constants $p\ge 2$ and $\mu >0$, and $\ell := \inf_{x\in C} \|\nabla f(x)\|^* > 0.$ 
Then~\eqref{eq.unif.convex.set.implied} implies that for $x\in\dom(\Psi), g:=\nabla f(x)$ and $s = \argmin_{x\in C} \ip{g}{x}$ we have
\begin{equation}\label{eq.key.step}
\gap(x,g) = \ip{g}{x-s} \ge \frac{\mu}{p}\|g\|^*\cdot \|x-s\|^p \ge 
\frac{\ell\mu}{p}\cdot \|x-s\|^p.
\end{equation}
Thus proceeding as in part (a),~\eqref{eq.unif.smooth.implied} implies that~\eqref{eq.growth.cond} holds for $q$ and $r=q/p$ with $M = 
L\cdot \left({p}/{(\ell \mu)}\right)^{\frac{q}{p}}$ 

\end{itemize}
\end{proof}

Proposition~\ref{prop.suff.gral} and Theorem~\ref{thm.main} yield the following norm-dependent convergence results for two classes of problems that extend the two classes discussed in Section~\ref{sec.2.1}.


First, suppose $\nabla f$ is $\nu$-H\"older continuous on $\dom(\Psi)$ for some $\nu \in (0,1]$ and $\Psi$ is $p$-uniformly convex on $\dom(\Psi)$ for some $p \ge 2$.  Then Proposition~\ref{prop.suff.gral} implies that the $(q,r)$-growth property~\eqref{eq.growth.cond} holds for $q = \nu+1$ and 
$r = (\nu+1)/p$ with $M = L\cdot \left(p/\mu\right)^{\frac{q}{p}}$.  The case $r = 1$, that is, $\nu=1$ and $p=2$ was already discussed in Section~\ref{sec.2.1}.  Hence we focus on the case $r < 1,$ that is, $p > \nu+1$.  
Under the mild assumption that $\gap_0\le M$, inequality~\eqref{eq.conv.gral} implies that
\begin{align}\label{eq.new.conv}
\gap_k  
& 
\le \left(\frac{1}{M^\frac{p-\nu-1}{p\nu}}+\frac{p-\nu-1}{p(\nu+1)}\cdot \frac{k}{M^\frac{1}{\nu}}\right)^\frac{p\nu}{\nu+1-p} 
\le \left(\frac{p(\nu+1) \cdot M^{\frac{1}{\nu}}}{p-\nu-1}\cdot 
 \frac{1}{k}
 \right)^\frac{p\nu}{p-\nu-1}. 
\end{align}
The above $\Oh(1/k^\frac{p\nu}{p-\nu-1})$ convergence rate is new for $p>2$ and extends some of the results considered in~\cite{Ghad19,Nest18} that only consider the cases when $p=2$ or when $\dom(\Psi)$ is bounded.

Second, if $\Psi=\delta_C$, $\nabla f$ is $\nu$-H\"older continuous and bounded away from zero on $C$, and $C\subseteq\R^n$ is a closed and $p$-uniformly set then Proposition~\ref{prop.suff.gral}(b) implies that the $(q,r)$-growth property~\eqref{eq.q=2.r=1} holds for $q=\nu+1, r = q/p$ with $M = L\cdot \left({p}/{(\ell\mu)}\right)^{\frac{q}{p}}$ where $L,\mu,$ and $\ell$ are respectively the Lipschitz continuity constant of $\nabla f$, the strong convexity constant of $C$ and the lower bound on $\|\nabla f(x)\|^*$ on $C$.  Again 
the case $r = 1$, that is, $\nu=1$ and $p=2$ was already discussed in Section~\ref{sec.2.1} and thus we focus on the case $r < 1.$  
Under the mild assumption that $\gap_0\le M$, inequality~\eqref{eq.conv.gral} again implies~\eqref{eq.new.conv} albeit for a different constant $M$. Hence we extend the sublinear convergence result shown in~\cite[Theorem 2.2]{Kerd21} that considers only the special case $\nu = 1$ and $p > 2$.  For these values of $\nu$ and $p$  we get
\[
\gap_k  \le 
\left(\frac{2Mp}{p-2}\cdot\frac{1}{k}\right)^\frac{p}{p-2} 
\]
and thus recover the $\Oh(1/k^{\frac{p}{p-2}})$ sublinear convergence result in~\cite[Theorem 2.2]{Kerd21} albeit with a slightly simpler and sharper constant.

\section{Weak growth property and suboptimality gap}
\label{sec.weak}
Consider the setting of Section~\ref{sec.main}.
Define the {\em suboptimality gap} function $\subopt:\dom(\Psi) \rightarrow \R$ as follows
\[
\subopt(x):=f(x) + \Psi(x) - \min_y\left\{f(y) + \Psi(y)\right\}.
\]
The weak duality property between~\eqref{eq.problem} and its dual~\eqref{eq.dual} implies that for all $(x,u) \in \dom(\Psi)\times \dom(f^*)$ 
\[
\gap(x,u) \ge \subopt(x).
\] 
Therefore any upper bound on the duality gap automatically holds for the suboptimality gap as well.  This section takes this observation further:   Theorem~\ref{thm.main.redux} and Theorem~\ref{prop.conv.gral.redux} provide analogous statements to
 Theorem~\ref{thm.main} and Theorem~\ref{prop.conv.gral}
 for the suboptimality gap by relying on the following relaxed version of~\eqref{eq.growth.cond}.   Suppose $q >1 $ and $r \in [0,1]$.  We shall say that the triple $(\D,\gap,\subopt)$ satisfies the {\em $(q,r)$ weak growth property} if there exists a finite constant $M >0$ such that for all $x \in \dom(\Psi), \, g:=\nabla f(x),$ and $s\in\partial \Psi^*(-g)$ 
\begin{equation}\label{eq.weak.growth.cond}
\D(x,s,\theta) \cdot \subopt(x)^{1-r} \le \frac{M\theta^{q}}{q}\cdot {\gap}(x,g) \text{ for all } \theta \in [0,1].
\end{equation}
Since $\gap(x,g) \ge \subopt(x) \ge 0$, it is immediate that the weak growth property~\eqref{eq.weak.growth.cond} holds whenever the growth property~\eqref{eq.growth.cond} holds.
 
The next two results concern the sequence of suboptimality gaps $$\subopt_k:=\subopt(x_k), \; k=0,1,\dots$$ for the iterates generated by Algorithm~\ref{algo.CG}.  We omit the proofs of Theorem~\ref{thm.main.redux} and Theorem~\ref{prop.conv.gral.redux} below as they are straightforward extensions of the proofs of Theorem~\ref{thm.main} and Theorem~\ref{prop.conv.gral}.

\begin{theorem}\label{thm.main.redux} Suppose Algorithm~\ref{algo.CG} chooses $\theta_k\in [0,1]$ via~\eqref{eq.stepsize} and  $q > 1$ and  $r\in [0,1]$ are such that
$(\D,\gap,\subopt)$ satisfy the $(q,r)$ weak growth property~\eqref{eq.weak.growth.cond} for some finite $M>0$. 
Then for $k=0,1,\dots$
\[
\subopt_{k+1} \le
\subopt_k\left(1-\frac{q-1}{q} \cdot \min\left\{1,\left(\frac{\subopt_k^{1-r}}{M}\right)^\frac{1}{q-1}\right\}\right). 
\]
When $r=1$ we have linear convergence
\[
\subopt_k\le\subopt_0\left(1-\frac{q-1}{q}\cdot
\min\left\{1,\frac{1}{M^{\frac{1}{q-1}}}\right\}
\right)^k.
\]
When $r\in[0,1)$ we have an initial linear convergence regime
\[
\subopt_{k} \le \subopt_0\left(1-\frac{q-1}{q} \right)^{k}, \; k=0,1,2,\dots,k_0
\]
where $k_0$ is the smallest $k$ such that 
$\subopt_k^{1-r} \le M$.  Then for $k\ge k_0$ we have a sublinear convergence regime
\[
\subopt_k\le\left(\subopt_{k_0}^\frac{r-1}{q-1}+\frac{1-r}{q}\cdot\frac{1}{M^{\frac{1}{q-1}}} \cdot (k-k_0)\right)^\frac{q-1}{r-1}.
\]

\end{theorem}

\begin{theorem}\label{prop.conv.gral.redux}
Suppose $c, \rho\in (0,1)$ are such that $c+\rho > 1$ and Algorithm~\ref{algo.CG} chooses $\theta_k\in [0,1]$ so that $\rho\cdot \hat \theta_k \le \theta_k \le \hat \theta_k$ for $\hat \theta_k$ as in~\eqref{eq.stepsize.relax}. 
Suppose $q > 1$ and  $r\in [0,1]$ are such that
 $(\D,\gap,\subopt)$ satisfy the $(q,r)$ weak growth property~\eqref{eq.weak.growth.cond} for some finite $M>0$.
Then for $k=0,1,\dots$
\[
\subopt_{k+1}\le 
\subopt_k\left(1-(c+\rho -1)\cdot\min\left\{1,\left(\frac{q(1-c)}{M} \cdot \subopt_k^{1-r}\right)^{\frac{1}{q-1}} \right\}\right).
\]
When $r=1$ we have linear convergence
\[
\subopt_k\le\subopt_0\left(1-(c+\rho -1)\cdot\min\left\{1,\left(\frac{ q(1-c)}{M}\right)^\frac{1}{q-1}\right\} \right)^k.
\]
When $r\in[0,1)$ we have an initial linear convergence regime
\[
\subopt_{k} \le \subopt_0\left(1-(c+\rho-1)\right)^{k}, \; k=0,1,2,\dots,k_0
\]
where $k_0$ is the smallest $k$ such that $\subopt_k^{1-r} \le \frac{M}{q(1-c)}$.  Subsequently for $k\ge k_0$ we have a sublinear convergence regime
\[
\subopt_k\le\left(\subopt_{k_0}^\frac{r-1}{q-1}+\frac{(1-r)(c+\rho -1)}{q-1}\cdot\left(\frac{q(1-c)}{ M}\right)^\frac{1}{q-1}\cdot (k-k_0) \right)^\frac{q-1}{r-1}.
\]
\end{theorem} 
Proposition~\ref{prop.Loja} below provides a nice complement to Proposition~\ref{prop.suff.gral}(b).  Proposition~\ref{prop.Loja} describes a class of problems that 
satisfies the weak growth property similar to that in Proposition~\ref{prop.suff.gral}(b).  However, instead of assuming that $\nabla f$ is bounded away from zero in $C$, Proposition~\ref{prop.Loja} relies on the following type of H\"olderian error bound as in~\cite{Kerd21,XY18}.  Suppose  $C\subseteq \R^n$ is a closed convex set,  $f^\star :=\min_{x\in C} f(x)$ is finite, $X^\star := \{x\in C: f(x) = f^\star\}$ is nonempty.  Suppose $\R^n$ is endowed with a norm $\|\cdot\|$ and $\gamma \in [0,1]$.
We shall say that $f$ satisfies the {\em $\gamma$-H\"olderian error bound}  on $C$ if there exists a finite constant $K > 0$ such that for all $x\in C$ 
\begin{equation}\label{eq.heb}
\min_{y\in X^\star} \|x-y\|\le K \cdot (f(x) - f^\star)^\gamma. 
\end{equation}
The H\"olderian error bound~\eqref{eq.heb} is known to hold generically when $C$ is a compact set and $f$ is subanalytic, as established by {\L}ojasiewicz~\cite{Loja65}, and further developed by Kurdyka~\cite{Kurd98} and Bolte et al.~\cite{BoltDL07}.

\begin{proposition}\label{prop.Loja} The $(q,r)$ weak growth property~\eqref{eq.weak.growth.cond} holds for $r = \gamma q/p$ if
$\Psi = \delta_C$ for some closed and $p$-uniformly convex set $C\subseteq \R^n$, $f$ is $q$-uniformly smooth on $C$, and the $\gamma$-H\"olderian error bound~\eqref{eq.heb} holds for $\gamma \in [0,1]$.\end{proposition}
\begin{proof} Suppose~\eqref{eq.unif.smooth} holds for some constants $q\in(1,2]$ and $L >0$,~\eqref{eq.unif.convex.set} holds  for some constants $p\ge 2$ and $\mu >0$, and~\eqref{eq.heb} holds  for some constants $\gamma \in [0,1]$ and $K >0$. Then~\eqref{eq.unif.convex.set.implied} implies that for all $x\in C$ and $g:=\nabla f(x)$ we have
\begin{equation}\label{eq.gap.g}
\gap(x,g) = \ip{g}{x-s} \ge \frac{\mu}{p}\|g\|^*\cdot \|x-s\|^p.
\end{equation}
Next, the convexity of $f$, Cauchy-Schwarz inequality, and~\eqref{eq.heb}  imply that for $x^\star:=\argmin_{y\in X^\star} \|x-y\|$
\begin{equation}\label{eq.g.bound}
\subopt(x) = f(x) - f(x^\star)\le \ip{g}{x-x^\star} \le \|g\|^*\cdot\|x-x^\star\| \le K\cdot \|g\|^*\cdot \subopt(x)^\gamma.
\end{equation}
From~\eqref{eq.gap.g} and~\eqref{eq.g.bound} it follows that
\[
 \|x-s\|^p \cdot\subopt(x)^{1-\gamma}\le \frac{pK}{\mu} \cdot \gap(x,g).
\]
The latter inequality,~\eqref{eq.unif.smooth.implied}, and $\subopt(x) \le \gap(x,g)$ imply that
\begin{align*}
\D(x,s,\theta) \cdot \subopt(x)^{1-\frac{\gamma q}{p}}  
&= \D_f(x+\theta(s-x),x) \cdot \subopt(x)^{1-\frac{\gamma q}{p}} 
\\ & \le\frac{L}{q}\cdot \theta^q \cdot \|s-x\|^q \cdot \subopt(x)^{\frac{(1-\gamma)q}{p}}\cdot \gap(x,g)^{1-\frac{q}{p}}
\\ & =\frac{L}{q}\cdot \theta^q \cdot \left(\|s-x\|^p \cdot \subopt(x)^{1-\gamma}\right)^\frac{q}{p}\cdot \gap(x,g)^{1-\frac{q}{p}}
\\
&\le  L\cdot\left(\frac{pK}{\mu}\right)^\frac{q}{p} \cdot\frac{\theta^q}{q} \cdot
\gap(x,g).
\end{align*}
Therefore~\eqref{eq.weak.growth.cond} holds for $r = \gamma q/p$ and 
$M =  L\cdot\left(pK/\mu\right)^{q/p}.$
\end{proof}

Proposition~\ref{prop.Loja} and Theorem~\ref{thm.main.redux} yield the following norm-dependent convergence results for a class of problems previously considered in~\cite{KerdL21,XY18}.

Suppose $\Psi=\delta_C$ for $C\subseteq \R^n$ closed and $p$-uniformly convex, $\nabla f$ is $\nu$-H\"older continuous on $C$, and $f$ satisfies the $\gamma$-H\"olderian error bound on $C$. Then Proposition~\ref{prop.Loja} implies that the $(q,r)$ weak growth property holds for $q=\nu+1, r = \gamma q/p$ with $M =  L\cdot\left(pK/\mu\right)^{q/p}$ where $L,\mu,$ and $K$ are respectively the H\"older continuity constant of $\nabla f$, the $p$-convexity constant of $C$, and the $\gamma$-H\"olderian error bound constant of $f$ on $C$.

When $\nu=1,\;p=2,$ and $\gamma = 1$ we have $q=2,\;r=1,$ and $M = \frac{2LK}{\mu}$.  Thus Theorem~\ref{thm.main.redux} yields the linear convergence rate
\[
\subopt_k \le \subopt_0\left(1-\frac{1}{2}\cdot\min\left\{1,\frac{\mu}{2LK}\right\} \right)^k.
\]
We thus recover the linear convergence result in~\cite[Theorem 1]{XY18} with a sharper rate.
For all other values of $\nu,p,\gamma$ we have $r = \gamma q/p < 1$. To simplify notation and exposition for this case,  assume that $\subopt_0 \le M$.  Thus Theorem~\ref{thm.main.redux} yields the sublinear rate
\begin{align*}
\subopt_k 
& \le\left(\frac{1}{M^\frac{1-r}{q-1}}+\frac{1-r}{q}\cdot\frac{k}{M^{\frac{1}{q-1}}}\right)^\frac{q-1}{r-1}
\le \left(\frac{(\nu+1) \cdot M^{\frac{1}{\nu}}}{1-\gamma (\nu+1)/p}\cdot \frac{1}{k}\right)^\frac{\nu}{1-\gamma (\nu+1)/p}.
\end{align*}
Hence we extend the sublinear convergence results in~\cite[Theorem 1]{XY18} and in~\cite[Theorem 2.9]{Kerd21}.  Both~\cite{XY18} and~\cite{Kerd21} consider  only the special case $\nu = 1$. In this special case we get $r = 2\gamma/p$ and thus
\[
\subopt_k  \le 
\left(\frac{2M}{1-2\gamma/p}\cdot \frac{1}{k}\right)^\frac{1}{1-2\gamma/p}
\]
In particular, we recover the $\Oh\left(1/k^{\frac{1}{1-\gamma}}\right)$
sublinear convergence rate in~\cite[Theorem 1]{XY18} which only considers the case $p=2$.  We also recover the $\Oh\left(1/k^{\frac{1}{1-2\gamma/p}}\right)$ sublinear convergence result in~\cite[Theorem 2.9]{Kerd21} which only considers the case $\gamma \in [0,1/2]$.  Like both~\cite{Kerd21} and~\cite{XY18}, we also recover the $\Oh\left(1/k^2\right)$  
sublinear convergence result in~\cite[Theorem 2]{GarbH15} which only considers the case $p=2$ and $\gamma = 1/2$.
In all cases we get  slightly simpler and sharper constants.

\section{Local growth property}
\label{sec.final}
We next describe an extension of the main developments in Section~\ref{sec.main} 
to a {\em local} setting.  The same extension applies to the developments in Section~\ref{sec.weak}.

Consider the setting of Section~\ref{sec.main}.  Suppose $q>1$ and $r\in [0,1]$.  We shall say that the pair $(\D,\gap)$ satisfies the $(q,r)$ {\em local growth property} if there exist constants $\epsilon > 0$ and $M > 0$ such that for all $x \in \dom(\Psi), \, g:=\nabla f(x),$ and $s := \partial \Psi^*(-g)$ that satisfy $\gap(x,g) < \epsilon$ we have
\begin{equation}\label{eq.local.growth.cond}
\D(x,s,\theta) \le \frac{M\theta^{q}}{q}\cdot {\gap}(x,g)^r \text{ for all } \theta \in [0,1].
\end{equation}
It is straightforward to see that local versions of Theorem~\ref{thm.main} and Theorem~\ref{prop.conv.gral} also hold if we replace the growth property~\eqref{eq.growth.cond} with the local growth property~\eqref{eq.local.growth.cond} provided $\gap_0 = \gap(x_0,g_0) < \epsilon$.   It is also evident that the local growth property~\eqref{eq.local.growth.cond} is less stringent than the growth property~\eqref{eq.growth.cond}.  In particular, the natural local version of Proposition~\ref{prop.suff.gral} holds if the uniform smoothness and uniform convexity assumptions hold on a neighborhood of an optimal solution to~\eqref{eq.problem}.  The following proposition, which is a straightforward extension of some arguments in~\cite{Kerd21},
describes another class of problems that satisfy the local growth condition~\eqref{eq.local.growth.cond}.  Proposition~\ref{prop.suff.gral.local} provides a counterpart to Proposition~\ref{prop.suff.gral}(b) that relaxes the assumption of $p$-uniform convexity of $C$ and boundedness away from zero of $\nabla f$ on $C$ to the far less stringent local scaling inequality~\eqref{eq.local.scaling} that only involves $\nabla f(x^\star)$ where $x^\star \in C$ is a minimizer of~\eqref{eq.problem}.  This relaxed assumption comes at the expense of downgrading the growth condition for $r = q/p$ to the local growth condition for $r = q(q-1)/(p(p-1))$.  As noted in~\cite[Remark 2.6]{Kerd21}, in the special case when $q=p=2$ there is no downgrade in the value of $r$. For all other values of $q\in(1,2]$ and $p\ge 2$ the downgrade factor  $(q-1)/(p-1)$ on $r$ is the price we pay for relaxing the uniform convexity and boundedness assumptions to the less stringent local scaling assumption.

\begin{proposition}\label{prop.suff.gral.local}
Suppose $\Psi = \delta_C$ for a closed convex set $C\subseteq\R^n$,  $x^\star \in C$ is a minimizer of~\eqref{eq.problem}, $\nu \in (0,1]$, and $p \ge 2$.  The local growth property~\eqref{eq.local.growth.cond} holds for $\epsilon = 1, \; q=1+\nu,$ and $r=q(q-1)/(p(p-1))$ if $\nabla f$ is $\nu$-H\"older continuous 
on $C$ and there exists $\sigma > 0$ such that the following local scaling inequality holds for all $x\in C$
\begin{equation}\label{eq.local.scaling}
\ip{\nabla f(x^\star)}{x-x^\star} \ge \sigma\cdot \|x-x^\star\|^p.
\end{equation}
\end{proposition}
\begin{proof}
From~\eqref{eq.local.scaling} and the convexity of $f$ it follows that for $x\in C$ and $g = \nabla f(x)$ 
\begin{equation}\label{eq.gap.1}
\gap(x,g) \ge f(x) - f(x^\star) \ge \ip{\nabla f(x^\star)}{x^\star-x} \ge \sigma\cdot \|x-x^\star\|^p.
\end{equation}
In addition, for $s = \argmin_{y\in C} \ip{g}{y}= \argmin_{y\in C} \ip{\nabla f(x)}{y}$ the local scaling inequality~\eqref{eq.local.scaling} implies that
\begin{equation}\label{eq.second}
\ip{\nabla f(x^\star)-\nabla f(x)}{s-x^\star} \ge \ip{\nabla f(x^\star)}{s-x^\star}\ge \sigma \cdot \|s-x^\star\|^p.
\end{equation}
Thus the $\nu$-H\"older continuity of $\nabla f$ and H\"older's inequality imply that for some finite H\"older continuity constant $L > 0$
\[
L \|x^\star - x\|^\nu \cdot \|s-x^\star\| 
\ge \sigma\cdot\|s-x^\star\|^p.
\]
Hence $\sigma \cdot \|s-x^\star\|^{p-1} \le L\|x^\star - x\|^\nu$ and thus~\eqref{eq.gap.1} yields
\begin{equation}\label{eq.gap.2}
\gap(x,g) \ge\sigma \cdot \|x-x^\star\|^p \ge \sigma \cdot \left(\frac{\sigma}{L}\right)^\frac{p}{\nu}\|s-x^\star\|^{\frac{p(p-1)}{\nu}}
\end{equation}
Let $\epsilon := 1$.  From~\eqref{eq.gap.1} and~\eqref{eq.gap.2} it follows that for $\gap(x,g)< \epsilon$
\begin{align*}
\|s-x\| &\le \|x^\star-x\| + \|s-x^\star\| \\
&\le 
\frac{1}{\sigma^\frac{1}{p}} \cdot
\gap(x,g)^\frac{1}{p}
+\frac{1}{\sigma^\frac{\nu}{p(p-1)}}\cdot
\left(\frac{L}{\sigma}\right)^\frac{1}{p-1}
\cdot
\gap(x,g)^\frac{\nu}{p(p-1)}
 \\
&\le \left(\frac{1}{\sigma^\frac{1}{p}}+\frac{1}{\sigma^\frac{\nu}{p(p-1)}}\cdot\left(\frac{L}{\sigma}\right)^\frac{1}{p-1}\right)\cdot \gap(x,g)^\frac{\nu}{p(p-1)}.
\end{align*}
The last step above holds because $\gap(x,g)< \epsilon = 1$ and $\nu/(p-1) \in(0,1]$ imply that $\gap(x,g)^\frac{1}{p} \le 
\gap(x,g)^\frac{\nu}{p(p-1)}$.
Therefore~\eqref{eq.unif.smooth.implied} implies that for $\gap(x,g)< \epsilon$
\begin{align*}
\D(x,s,\theta) & = \D_f(x+\theta(s-x),x) \\
& \le \frac{L\theta^{1+\nu}}{1+\nu}\cdot  \|s-x\|^{1+\nu} \\
& \le  L\cdot 
\left(\frac{1}{\sigma^\frac{1}{p}}+\frac{1}{\sigma^\frac{\nu}{p(p-1)}}\cdot\left(\frac{L}{\sigma}\right)^\frac{1}{p-1}\right)
^{1+\nu} \cdot\frac{\theta^{1+\nu}}{1+\nu}\cdot \gap(x,g)^\frac{\nu(1+\nu)}{p(p-1)}.
\end{align*}
Therefore whenever $\gap(x,g)<\epsilon$ the local growth property~\eqref{eq.local.growth.cond} holds for $q=1+\nu$ and $r = q(q-1)/(p(p-1))$ with 
$$M =  L\cdot 
\left(\frac{1}{\sigma^\frac{1}{p}}+\frac{1}{\sigma^\frac{\nu}{p(p-1)}}\cdot\left(\frac{L}{\sigma}\right)^\frac{1}{p-1}\right)
^{1+\nu}. 
$$

\end{proof}

\end{document}